\documentclass[12pt]{article}
\usepackage{amsmath}
\usepackage{amsthm}
\usepackage{amsfonts}
\usepackage{amssymb}
\usepackage{latexsym}

\usepackage{amscd}
\usepackage{mathrsfs}
\usepackage{stmaryrd}
\usepackage{pst-node}
\usepackage{tikz-cd}

\newcommand{\laplace}{\triangle}
\newcommand{\del}{\nabla}

\newcommand{\ulx} {{\underline x}}

\newcommand{\bfe} {{\mathbf e}}

\newtheorem{thm}{\bf Theorem}[section]

\newtheorem{prop}[thm]{\bf Proposition}

\newtheorem{defn}[thm]{\bf Definition}

\AtEndDocument{\bigskip{\footnotesize
		\textsc{Macao Center for Mathematical Sciences, Macau University of Science and Technology, Macau, China}\par
		\textit{E-mail address: }\texttt{wxmai@must.edu.mo}\par
		\medskip
		\textsc{School of Mathematics (Zhuhai), Sun Yat-Sen University, Guangdong, China}\par
		\textit{E-mail address: }\texttt{shaogk@mail.sysu.edu.cn}
}}

\title{On the Bergman kernel in weighted monogenic Bargmann-Fock spaces 
	\footnotetext{The first author was supported in part by NSFC Grant No. 11901594. The second author was supported by NSFC Grant No. 12001549, Guangdong Basic and Applied Basic Research Foundation Grant No. 2019A1515110250 and Science and Technology Projects in Guangzhou Grant No. 201904010436.}}

\author{Weixiong Mai, Guokuan Shao}

\date{}

\begin{document}
	\maketitle
	\begin{center}
		\begin{minipage}{120mm}
			\begin{center}{\bf Abstract}\end{center}
			{ In this paper, we study the Bergman kernel $B_\varphi(x,y)$ of generalized Bargmann-Fock spaces in the setting of Clifford algebra. The versions of $L^2$-estimate method and weighted subharmonic inequality for Clifford algebra are established. Consequently we show the existence of $B_\varphi(x,y)$ and then give some estimates on- and off- the diagonal. As a by-product, we also obtain an upper estimate of the weighted harmonic Bergman kernel.\\
				
			}
		\end{minipage}
	\end{center}
	\tableofcontents
	
	\section{Introduction}
The Bergman kernel plays a crucial role in the study of complex analysis and geometry. The asymptotic expansions and estimates for the Bergman kernel has been studied in many different context of domains on $\mathbb C^n$, as well as line bundles on compact complex manifolds started in the paper of Tian \cite{Tian} (cf. e.g. \cite{BFG, Zelditch, MM, Schuster-Varolin,HM}). It is not only of independent interest but also has significant applications in some important problems such as existence of K\"ahler metrics (cf.\cite{D}), Beurling-Landau density (cf. e.g. \cite{Berndtsson, Ortega-Cerda-Seip,Lindholm, Ameur-Ortega}) and equidistribution for zeros of holomorphic sections or Fekete points (cf. e.g. \cite{SZ, CMN, Lev-Ortega}).
	
	In one and several complex variables, a lot of effort has been made to estimate the Bergman kernel on- and off-diagonal. For instance, in the case of one complex variable M. Christ \cite{Christ} and J. Marzo and J. Ortega-Cerd\`a \cite{Marzo-Ortega} showed pointwise estimates for the Bergman kernel of the weighted Fock space $F^2_{\varphi}(\mathbb C)$ under the assumption that $\laplace \varphi$ is a doubling measure, and under a similar assumption H. Delin \cite{Delin} and H. Lindholm \cite{Lindholm} extended this result to several complex variables. 

	In this paper we will focus on the estimates for the Bergman kernel of generalized Bargmann-Fock spaces in the setting of Clifford algebra. It would be natural to expect an analogous result in Clifford analysis since the monogenic function theory in the Clifford algebra setting can be considered as a generalization of the holomorphic function theory to $\mathbb R^{n+1},$ which is different from the case in several complex variables. Due to the importance of the Bargmann-Fock space, there are various generalizations in Clifford analysis (see e.g. \cite{Gong-Leong-Qian, Pena-Sabadini-Sommen, Mourao-Nunes-Qian, Bernstein-Schufmann}). 

Note that 
$\mathcal{A}_n$ is a real Clifford algebra over $\mathbb R^{n+1}$, where $\mathbb R^{n+1}=\{x=x_0\bfe_0+\ulx:x_0\in\mathbb R, \ulx=x_1\bfe_1+\cdots +x_n\bfe_n\in \mathbb R^n\},$ based on the multiplication rules
	\begin{align*}
		&\bfe_j\bfe_k + \bfe_k\bfe_j =0, \quad j\neq k,\\
		&\bfe_j^2=-1, \quad j=1,2,...,n,
	\end{align*}
	and $\bfe_0=1$ as the unit element. Denote by 
\begin{equation*}
	F_\varphi^2(\mathbb R^{n+1}, \mathcal{A}_n):=  \{u\text{ is left monogenic on } \mathbb R^{n+1}: ||u||_{\varphi}^2= \int_{\mathbb R^{n+1}} |u(x)|^2e^{-2\varphi(x)}dx<\infty\}
	\end{equation*}
the Bargmann-Fock space associated with the weighted function $\varphi$, which consists of left monogenic functions (see \S 2 for details), i.e., for $u\in F^2_\varphi(\mathbb R^{n+1},\mathcal{A}_n)$ we have
	\begin{align*}
		Du =\sum_{j=0}^n\bfe_j\frac{\partial u}{\partial x_j}=0,\quad \text{for $u$ taking values in }\mathcal{A}_n, 
	\end{align*}
 where $D$ is the Dirac operator. In many cases, the study of the Bargmann-Fock space in the Clifford algebra setting is based on $\varphi(x)=|x|^2.$ 
	
	Generally, in the setting of complex analysis, to obtain the upper bound estimate for the Bergman kernel on the diagonal, one needs a weighted mean-value inequality for holomorphic functions in $F_\varphi^2(\mathbb C^n),\ n\geq 1$. It is known that J. Ortega-Cred\`a and K. Seip proved the inequality for functions in $F_\varphi^2(\mathbb C)$ in \cite{Ortega-Cerda-Seip} which is based on the properties of holomorphic functions, and  the higher dimensional case was proved in \cite{Delin}( see also \cite{Lindholm}) by using H\"ormander's $L^2$-estimate for the $\overline \partial$ operator. In the lower bound estimate for the Bergman kernel on the diagonal,  H\"ormander's $L^2$-estimate for the $\overline \partial$ operator also plays a crucial role (see e.g. \cite{Lindholm}). 

    Unfortunately, neither the above arguments used in the upper bound estimate nor those used in the lower bound estimate could be employed in our case. 
    Unlike the complex case we do not even have an explicit formula for the Bergman kernel in $F_{|x|^2}^2(\mathbb R^{n+1},\mathcal{A}_n)$ because of the non-commutativity of Clifford algebras. In fact, many elementary facts in the complex case are no longer true in the Clifford algebra setting. The main diffculties in our study are also caused by the non-commutativity of Clifford algebras. In fact, the properties that products of holomorphic functions are also holomorphic and constructing holomorphic functions by complex exponential function used in \cite{Ortega-Cerda-Seip} are not true for monogenic functions. To the authors' knowledge, the generalization of H\"ormander's $L^2$-estimate in Clifford analysis is given by Y. Liu, Z. Chen and Y. Pan in \cite{Liu-Chen-Pan}. Note that the Dirac opeprator $D$ is viewed as a generalization of $\overline \partial$. The problem is that, due to the treatment used in \cite{Liu-Chen-Pan}, the authors imposed strong conditions on $\varphi$ in \cite{Liu-Chen-Pan}, which causes that this result could not be directly applied in our study. Nevertheless, Liu-Chen-Pan's arguments and those techniques used in \cite{Lindholm, Ortega-Cerda-Seip} still shed a light on our study. 

To overcome the above obstacles, we need to introduce some new technical treatments. To generalize a weighted mean-value inequality for functions in $F_\varphi^2(\mathbb R^{n+1}, \mathcal{A}_n),$ we will adapt the so-called Moser's iterative method (see e.g. \cite{Han-Lin}) from the theory of elliptic partial differential equations. To obtain the lower bound estimate for the Bergman kernel of $F_\varphi^2(\mathbb R^{n+1}, \mathcal{A}_n)$, we need a generalization of H\"ormander's $L^2$-estimate on $D.$ Indeed, we will improve the result given in \cite{Liu-Chen-Pan} where we remove their imposed conditions on $\varphi$. The key is that some nice properties of vector-valued functions will be used in our proof. In the rest of the paper, we abusively use $C$ to denote positive constants, where $C$ is not necessary to be the same in different theorems.
	
   Our main result is stated as follows.
	\begin{thm}\label{main}
		Given $\varphi\in C^2(\mathbb R^{n+1},\mathbb R),$ $0<m\leq \laplace \varphi(x)\leq M$ for all $x\in\mathbb R^{n+1}$ and {$|\del \varphi|\leq L|x|,$} where $m,M, L$ are positive constants and $\laplace=\sum_{j=0}^{n}\frac{\partial^2}{\partial x_j^2}$. Then the Bergman kernel $B_\varphi(x,w)$ of $F^2_\varphi(\mathbb R^{n+1},\mathcal{A}_n)$ exists. On the diagonal, there exist positive constants $C_1,C_2$ such that 
		\begin{align*}
			|B_\varphi(x, x)|\leq  C_1e^{2\varphi(x)}
		\end{align*}
		and if additionally, $\varphi$ is $2$-homongenous,
		\begin{align*}
			|B_\varphi (x, x)|\geq C_2 e^{2\varphi(x)}.
		\end{align*}
		Off the diagonal, i.e. $x\neq y$, there exist positive constants $C$ and $\alpha$ such that
		\begin{align*}
			|B_\varphi(x, y)|\leq C e^{\varphi(x)+\varphi(y)-\alpha|x-y|}.
		\end{align*}
	\end{thm}
	
	Note that the last estimate in the theorem is a generalization of Christ's result \cite{Christ}. 
	
	To prove Theorem \ref{main}, we will need two technical results. One is a variant of H\"ormander's $L^2$-estimate for the Dirac operator in the Clifford algebra setting, and the other is a weighted subharmonic inequality for functions in $F^2_\varphi(\mathbb R^{n+1},\mathcal{A}_n).$ Compared with the roles of $L^2$-method and mean-value inequality in estimates for Bergman kernels in complex geometry, the following three theorems are of importance to prove the above main result.
	Since the two results are indeed of independent interest, we present them in the following and give some remarks.
	
We consider $Du=f,$ which is analogous to the $\overline\partial$-equation in several complex variables, where $u,f$ take values in the Clifford algebra. Let $\Omega$ be a domain in $\mathbb R^{n+1}$. Denote by $L^2_\varphi(\Omega,\mathcal{A}_n)$ the space of all $L^2$ functions which take values in the Clifford algebra $\mathcal{A}_n$ with respect to the weight $\varphi$. 
We define the operator $D^*_\varphi$ acting on components of $\alpha = \sum_A \alpha_A e_A,$ where $\alpha_A\in C_0^\infty(\Omega,\mathbb R)$ (cf. Subsection 2.1). That is, $D^*_\varphi \alpha_A = -e^{2\varphi} \overline D(\alpha_A e^{-2\varphi})= 2\alpha_A \overline D\varphi-\overline D\alpha_A.$ 
	
	\begin{thm}\label{thm1}
		Suppose that $f\in L^2_\varphi(\Omega,\mathcal{A}_n).$ There exists $u\in L^2_\varphi(\Omega,\mathcal{A}_n)$ such that
		\begin{align}\label{dbar}
			Du=f
		\end{align}
		in the weak sense with 
		\begin{align}\label{cond1}
			||u||_\varphi^2 = \int_\Omega |u|^2 e^{-2\varphi}dx\leq C,
		\end{align}
		if
		\begin{align}\label{cond2}
			|(f,\alpha)_\varphi|^2 = \left|\int_{\Omega}\overline {\alpha(x)} f(x) e^{-2\varphi(x)} dx\right|^2\leq 2^{2n}C \sum_{A}||D_\varphi^*\alpha_A||^2_\varphi, 
		\end{align}
for all $\alpha=\sum_A \alpha_A e_A\in C_0^\infty(\Omega,\mathcal{A}_n)$.
		Conversely, if there exists $u\in L^2_\varphi(\Omega,\mathcal{A}_n)$ satisfying \eqref{dbar} with
		\begin{align*}
		    ||u||_\varphi^2 \leq 2^{2n}C,
		\end{align*}
then the inequality \eqref{cond2} holds true.
		Here $C$ is a positive constant.
	\end{thm}
	
	\begin{thm}\label{thm2}
		Given $\varphi\in C^2(\Omega,\mathbb R)$, and $\laplace \varphi\geq 0.$ Then for all $f\in L^2_\varphi(\Omega,\mathcal{A}_n)$ with $\int_\Omega \frac{|f|^2}{\laplace\varphi}e^{-2\varphi} dx <\infty,$ there exists a $u\in L^2_\varphi(\Omega,\mathcal{A}_n)$ such that 
		\begin{align*}
			Du=f
		\end{align*}
		with 
		\begin{align*}
			||u||_\varphi^2 = \int_\Omega |u(x)|^2 e^{-2\varphi(x)} dx \leq 2^{2n}  \int_\Omega \frac{|f(x)|^2}{\laplace \varphi(x)} e^{-2\varphi(x)} dx.
		\end{align*}
	\end{thm}
	Note that the above two theorems are the extensions of the corresponding theorems proved in Liu-Chen-Pan’s article \cite{Liu-Chen-Pan}.
	Since H\"ormander's $L^2$-method plays an important role in complex analysis and complex geometry, we would expect to find more applications of Theorems \ref{thm1} and \ref{thm2} in furture.

	\begin{thm}\label{thm3}
		For $u\in F_\varphi^2(\mathbb R^{n+1},\mathcal{A}_n)$, and $R\in (0,1)$, there exists $C_R>0$ such that the estimate holds 
		\begin{align*}
			|u(x)|e^{-\varphi(x)} \leq C_{R} \left(\int_{B_x(R)}|u(y)|^2e^{-2\varphi(y)}dy\right)^\frac{1}{2},
		\end{align*}
		where $\varphi\in C^2(\mathbb R^{n+1},\mathbb R)$ satisfies $|\del \varphi(x)|\leq {L|x|}$ with $L>0$ independent of $R$.  
	\end{thm}

		We note that complex analysis methods do not play a role in the proof of Theorem \ref{thm3}, and the main tool is Moser's iterative method. Therefore, we can partially generalize the above theorem to harmonic functions on $\mathbb R^{n+1},$ and the result is stated as follows. Define 
		\begin{align*}
			F^2_{\varphi,har}(\mathbb R^{n+1}) = \{u \text{ is harmonic on } \mathbb R^{n+1}: ||u||^2_{\varphi,har} = \int_{\mathbb R^{n+1}} |u(x)|^2e^{-2\varphi(x)}dx<\infty\},
		\end{align*}	
		and denote by $B_{\varphi,har}(x,y)$ the harmonic  Bergman kernel for $F^2_{\varphi,har}(\mathbb R^{n+1}).$
		\begin{thm}\label{thm4}
			For $u\in F_{\varphi,har}^2(\mathbb R^{n+1}),$ and $R\in (0,1)$, there exists $C_R>0$ such that the estimate holds 
			\begin{align*}
				|u(x)|e^{-\varphi(x)} \leq C_{R} \left(\int_{B_x(R)}|u(y)|^2e^{-2\varphi(y)}dy\right)^\frac{1}{2},
			\end{align*}
			where $\varphi\in C^2(\mathbb R^{n+1},\mathbb R)$ satisfies $|\del \varphi(x)|\leq L|x|$ with $L>0$ independent of $R$.  
			Moreover, there exists a constant $C$ such that
			\begin{align*}
				|B_{\varphi,har}(x,y)| \leq C e^{\varphi(x)+\varphi(y)}.
			\end{align*}
		\end{thm}
		
		The proof of Theorem \ref{thm4} could be immediately deduced from the proof of Theorem \ref{main} and Theorem \ref{thm3}.  In fact, the study of the harmonic Bergman kernel has been paid attention by some researchers in the last two decades. For instance, in \cite{Kang-Koo} H. Kang and H. Koo studied estimates of the harmonic Bergman kernel on smooth domains; in \cite{Englis1,Englis2} M. Engli\v{s} studied the asymptotic behaviors of the weighted harmonic Bergman kernels for Bermgan spaces on bounded domains with a positive weight, as well as the harmonic Bergman kernel for harmonic Bargmann-Fock space with the weight $|x|^2$. To  the authors' knowledge, Theorem \ref{thm4} should be new for studying  the weighted harmonic Bergman kernel. Moreover, the used methods in this paper would provide a new approach to estimate the weighted harmonic Bergman kernel.  It would be also very interesting to obtain the lower bound of $|B_{\varphi,har}(x,x)|$ and the analogue of Christ's result for $B_{\varphi,har}(x,y).$ 
	
	This paper is organized as follows. In \S 2 we will give some basic notations and concepts in Clifford analysis. In \S 3 we will prove the technical results: 1. the $L^2$-estimate for the Dirac operator; 2. the mean-value inequality for functions in $F_\varphi^2(\mathbb R^{n+1},\mathcal{A}_n).$ In \S 4 we will prove our main result.

	\section{Preliminaries}
	In this section we will give some basic notations and results in Clifford analysis.
	\subsection{Clifford algebra}
	Denote by $\mathcal{A}_n$ the algebra over the real number field generated by the basis $\{\bfe_0,\bfe_1,...,\bfe_n\}$ of $\mathbb R^{n+1}$ where the ${\bfe_j}$'s satisfy the relations 
	\begin{align*}
		\bfe_j\bfe_k+ \bfe_k\bfe_j=0,\quad j \neq k,\\
		\bfe_j^2 = -1, \quad j=1,2,...,n.
	\end{align*}
	$\mathcal{A}_n$ is a real Clifford algebra with the unit element $\bfe_0 =1.$  
	
	The element of $\mathcal{A}_n$ is called a Clifford number and is given by $x=\sum_{A}x_A \bfe_A$ where $A=\{1\leq j_1<j_2<\cdots<j_l\leq n\}$ runs over all ordered subsets of $\{1,...,n\}$, $x_A\in \mathbb R $ with $x_{\emptyset}=x_0,$ and $\bfe_A=\bfe_{j_1}\bfe_{j_2}\cdots \bfe_{j_l}$ with the unit element $\bfe_{\emptyset}=\bfe_0=1.$ 
	$[x]_0=x_0$ is the scalar part of the Clifford number $x$ and $[x]_A$ denotes the coefficient $x_A$ of the $\bfe_A$-component. 
	Similar to the complex conjugation, we define involutions $\bar\ $ for the real Clifford algebra $\mathcal{A}_n.$ Let
	\begin{align*}
		\overline x = \sum_A x_A \overline \bfe_A
	\end{align*}
	for $x\in\mathcal{A}_n$, where $\overline \bfe_A = (-1)^{\frac{|A|(|A|+1)}{2}}\bfe_A.$ By definition it is clear that $\overline \bfe_A \bfe_A = \bfe_A\overline \bfe_A =1.$ Moreover,
	\begin{align*}
		\overline {\lambda\mu} = \overline\mu\overline \lambda, \quad {\text{for all }} \lambda,\mu\in \mathcal{A}_n.
	\end{align*}
	Refering to \cite{Gilbert-Murray}, we see that $\mathcal{A}_n$ becomes a finite dimensional Hilbert space with the inner product
	\begin{align*}
		(x,y)_0= [x\overline y]_0 = \sum_A x_A y_A
	\end{align*}
	for all $x,y\in\mathcal{A}_n,$ and has the corresponding norm
	\begin{align*}
		|x|_0 = \sqrt{(x,x)_0} = \sqrt{\sum_A |x_A|^2}.
	\end{align*}
	To make the notation simplify, we will write the norm of $x\in\mathcal{A}_n,$ $|x|_0$ as $|x|$ when there is no confusion. 
	
	\begin{defn}\label{mono}
An $\mathcal{A}_n$-valued function $u(x)=\sum_A u_A(x)\bfe_A$ on an open subset $\Omega$ of $\mathbb R^{n+1}$ is said to be left-monogenic (resp. right-monogenic) if
	\begin{align*}
		D u= \sum_{j=0}^n \bfe_j\partial_j u = \sum_{j,A} \bfe_j\bfe_A\partial_ju_A =0 \left({\text {resp}.}\ uD=\sum_{j=0}^n \partial_j u\bfe_j=\sum_{j,A}\partial_ju_A\bfe_A\bfe_j=0\right)
	\end{align*}
	where $u_A(x)$ are real-valued functions, $\partial_j=\frac{\partial}{\partial x_j},0\leq j\leq n,$ and $D$ is the Dirac operator. 
\end{defn}
Note that $\overline D(Du)=\Delta u=0$ if $u$ is left-monogenic, which means that each component of a left-monogenic function $u$ is harmonic.
	A function that is both left- and right-monogenic is called a monogenic function. 
	It is noted that when $n=1,$ the Dirac operator $D=\partial_0+\partial_1 \bfe_1$ with $\bfe_1^2=-1,$ which coincides with the operator $\overline\partial$ in complex analysis. Therefore, the Dirac operator  in $\mathbb R^{n+1}$ could be considered as an analogue of $\overline \partial$.
	
	\subsection{Hilbert Clifford-Modules}
	In order to study spaces of $\mathcal{A}_n$-valued functions, we need an analogue of $L^2$ spaces. We have to introduce Clifford-modules since the elements of a Clifford algebra do not form a field. In the sequel definitions and properties will be stated for left $\mathcal{A}_n$-module and their duals, which are taken from the reference \cite{Brackx-Delanghe-Sommen}. Note that those in the right $\mathcal{A}_n$-module are similar.
	
	\begin{defn}\label{defmodule}
		Let $X$ be a unitary left $\mathcal{A}_n$-module, when $(X,+)$ is an abelian group and the mapping $(\lambda, f)\to \lambda f: \mathcal{A}_n\times X\to X $ is defined such that for any $\lambda,\mu\in \mathcal{A}_n$ and $f,g\in X$
		\begin{enumerate}
			\item $(\lambda+\mu)f = \lambda f+\mu f,$
			\item $\lambda \mu f= \lambda(\mu f),$
			\item  $\lambda(f+g) =\lambda f+\lambda g,$
			\item  $\bfe_0 f= f.$
		\end{enumerate}
	\end{defn}
	Note that $E$ is said to be a submodule of the unitary left $\mathcal{A}_n$-module $X,$ if $E$ is a non-empty subset of $X$ such that $E$ becomes a unitary left $\mathcal{A}_n$-module when restricting the module operation of $X$ to $E.$

	\begin{defn}
		If $X,Y$ are unitary left $\mathcal{A}_n$-modules, then $T:X\to Y$ is said to be a left $\mathcal{A}_n$-linear operator, if for any $f,g\in X$ and $\lambda \in\mathcal{A}_n$ the following holds
		\begin{align*}
			T(\lambda f+g) = \lambda T(f) + T(g).
		\end{align*}
	\end{defn}
	Denote by $L(X,Y)$ the set of all left $\mathcal{A}_n$-linear operator. In particular, if $Y=\mathcal{A}_n$, $L(X,\mathcal{A}_n)$ is called the algebraic dual of $X,$ which is denoted by $X^{*alg}.$ 
	
	An element $T\in X^{*alg}$ is bounded if there exist a semi-norm $p$ on $X$ and $c>0$ such that for all $f\in X$ 
	\begin{align*}
		|\langle T, f \rangle |\leq c\cdot p(f).
	\end{align*} 
	
	\begin{thm}\label{HB}
		Let $X$ be a unitary left $\mathcal{A}_n$-module with semi-norm norm $p,$ $Y$ a submodule of $X$ and $T$ a left $\mathcal{A}_n$-linear functional on $Y$ such that for some $c>0,$ the following holds
		\begin{align*}
			|\langle T, g\rangle|\leq c\cdot p(g),\quad {\text for \ all } \ g\in Y.
		\end{align*} 
		Then there exists a left $\mathcal{A}_n$-linear functional $\widetilde T$ on $X$ such that
		\begin{enumerate}
			\item $\widetilde T|_Y=T,$
			\item for some  $c^* > 0$, $|\langle \widetilde T,f\rangle|\leq c^*\cdot p(f),$ for all $f\in X.$
		\end{enumerate}
	\end{thm}
	
	We define an inner product on a unitary right $\mathcal{A}_n$-module as follows.
	\begin{defn}
		Let $H$ be a unitary right $\mathcal{A}_n$-module. Then a function $( \cdot,\cdot) :H\times H\to \mathcal{A}_n$ is an inner product on $H$ if for all $f,g,h\in H$ and $a\in \mathcal{A}_n,$
		\begin{enumerate}
			\item $( f, g+h ) = (f, g) + ( f, h),$
			\item $(f,g\lambda) =(f,g)\lambda,$
			\item $(f,g) = \overline{(g,f)},$
			\item $[(f,f)]_0\geq 0$ and $[(f,f)]_0=0$ if and only if $f=0,$
			\item $[(fa,fa)]_0\leq |a|^2[(f,f)]_0.$
		\end{enumerate}
		The accompanying norm on $H$ is $||f||^2=[(f,f)]_0.$
	\end{defn}
	We now have the following Cauchy-Schwarz inequality.
	\begin{prop}
		If $(\cdot,\cdot)$ is an inner product on a unitary right $\mathcal{A}_n$-module $H$ and $||f||^2=[(f,f)]_0$, then for all $f,g\in H,$ $$|(f,g)|\leq 2^n ||f|| ||g||.$$
	\end{prop} 
	Let $H$ be a unitary right $\mathcal{A}_n$-module equipped with an inner product $(\cdot,\cdot).$ It is called a right Hilbert $\mathcal{A}_n$-module if it is complete for the norm topology derived from the equipped inner product.
	
	In this paper we will also use the following Riesz representation theorem.
	\begin{thm}
		Let $H$ be a right Hilbert $\mathcal{A}_n$-module and $T\in H^{*alg}.$ Then $T$ is bounded if and only if there exists a (unique) element $g\in H$ such that for all $f\in H$,
		\begin{align*}
			T(f) := \langle T, f\rangle = (g,f).
		\end{align*}
	\end{thm}
	
	Denote by $L^2_\varphi(\mathbb R^{n+1},\mathcal{A}_n)$ the weighted $L^2$ space defined as
	\begin{align*}
		L^2_\varphi(\mathbb R^{n+1},\mathcal{A}_n) = \{u=\sum_{A}u_A\bfe_A:||u||_{\varphi}^2= \int_{\mathbb R^{n+1}} |u(x)|^2e^{-2\varphi(x)}dx<\infty\}
	\end{align*}
	equipped with the inner product
	\begin{align*}
		(f,g) = (g,f)_\varphi = \int_{\mathbb R^{n+1}} \overline {f(x)}g(x)e^{-2\varphi(x)}dx.
	\end{align*}
	Note that one can easily see that 
	\begin{align*}
	    ||f||_{\varphi}^2= [(f,f)]_0.
	\end{align*}

Consider now $u=u_0+\sum_{j=1}^nu_je_j$, where $u_j$ are real functions (such function $u$ is called a vector-valued function in the paper). It is easy to see that $u\overline u=[u\overline u]_0=|u|^2$. When $f=\lambda u$, where a constant $\lambda=\sum_A \lambda_Ae_A$, we have 
\begin{equation}\label{e-121}
\begin{split}
\|\lambda u\|_\varphi^2&=\int_{\mathbb R^{n+1}}|\lambda u(x)|^2e^{-2\varphi(x)}dx\\
&=\int_{\mathbb R^{n+1}}[\lambda u\overline{\lambda u}]_0 e^{-2\varphi(x)}dx=\int_{\mathbb R^{n+1}}[\lambda u\overline{u}\overline{\lambda}]_0 e^{-2\varphi(x)}dx\\
&=\int_{\mathbb R^{n+1}}|\lambda|^2|u|^2 e^{-2\varphi(x)}dx=|\lambda|^2\|u\|_\varphi^2.
\end{split}
\end{equation}

	The generalized Bargmann-Fock space consisting of left monogenic functions is defined as
	\begin{align*}
		F^2_\varphi(\mathbb R^{n+1}, \mathcal{A}_n) = \{u\text{ is left monogenic on } \mathbb R^{n+1}: ||u||_{\varphi}^2= \int_{\mathbb R^{n+1}} |u(x)|^2e^{-2\varphi(x)}dx<\infty\}.
	\end{align*}
	We would like to study the corresponding Bergman kernel, denoted by $B_\varphi(w,x)$, of $F_\varphi(\mathbb R^{n+1},\mathcal{A}_n)$. First of all, we need to show the existence of $B_\varphi(w,x)$. By the Riesz representation theorem, it suffices to show that the point-wise evaluation functional $T$ is bounded, i.e.,
	\begin{align*}
		|T(u)(x)|\leq C_x ||u||_\varphi,
	\end{align*}
	which will be proved in the sequel section.
	
	\section{Two technical results}
	To prove our main result, we need the technical results Theorem \ref{thm1}, Theorem \ref{thm2} and Theorem \ref{thm3}. The first two concern the $L^2$-estimate for the Dirac operator in the Clifford algebra setting, and the third is the mean value inequality for $F_\varphi^2(\mathbb R^{n+1},\mathcal{A}_n).$ 
	\subsection{$L^2$-estimate for the Dirac operator}
	A variant of H\"ormander's $L^2$-estimate for the Dirac operator would be the key to give the lower bound of the weighted Bergman kernel on the diagonal.

    Recall that the operator $D^*_\varphi$ acts on components of $\alpha = \sum_A \alpha_A e_A,$ where $\alpha_A\in C_0^\infty(\Omega,\mathbb R).$ That is, $D^*_\varphi \alpha_A = -e^{2\varphi} \overline D(\alpha_A e^{-2\varphi})= 2\alpha_A \overline D\varphi-\overline D\alpha_A.$ Instead of acting on $\alpha=\sum_A \alpha_A e_A$, $D^*_\varphi$ acting on components is a key ingredient in our proofs, which allows us to use some properties of vector-valued functions (cf. \eqref{e-121}). This is crucial in our proofs with which we can remove the conditions given in \cite{Liu-Chen-Pan}. 
	To prove Theorem \ref{thm1} and \ref{thm2}, we consider 
	\begin{align}\label{def1}
		\begin{split}
			(f,\alpha)_\varphi & = \int_\Omega \overline {\alpha(x)} f(x) e^{-2\varphi(x)} dx\\
			& = \int_\Omega (\sum_A \alpha_A(x) \overline {e_A}) Du(x)e^{-2\varphi(x)} dx\\
			& =\sum_A \overline{e_A}\int_\Omega \alpha_A(x) Du(x)e^{-2\varphi(x)} dx\\\\
			& = -\sum_A \overline {e_A}\int_\Omega [(\alpha_A(x)e^{-2\varphi(x)}) D]u(x) dx\\
			& = \sum_A \overline{e_A}\int_\Omega -\overline {e^{2\varphi(x)} \overline D(\alpha_A(x)e^{-2\varphi(x)})} u(x) e^{-2\varphi(x)} dx\\
			& = \sum_A\overline{e_A} (u,-e^{2\varphi} \overline D(\alpha_Ae^{-2\varphi}))_\varphi \\
			& = \sum_A \overline{e_A}(u,D_\varphi^* \alpha_A)_\varphi,
		\end{split}
	\end{align} 
where the forth equality just follows from integration by parts for each component ( one could also obtain this equality by using Stokes-Green theorem in the Clifford analysis setting \cite{Brackx-Delanghe-Sommen} ).

	\begin{proof}[{\rm \textbf{Proof of Theorem \ref{thm1}:}}]
		(Necessity) If there holds (\ref{dbar}) and (\ref{cond1}), then 
		\begin{align*}
			|(f,\alpha)_\varphi|^2 &= |(Du,\alpha)_\varphi|^2\\
			& = \left|\int_\Omega \overline {\alpha(x)} Du(x)e^{-2\varphi(x)} dx\right|^2\\
			& = \left| \sum_A \overline{e_A}\int_\Omega \alpha_A(x) Du(x) e^{-2\varphi(x)} dx \right|^2 \\
			& = \left| \sum_A \overline {e_A} (u,D^*_\varphi \alpha_A)_\varphi \right|^2 \\
			& \leq 2^{2n} ||u||_\varphi^2 \sum_A ||D^*_\varphi \alpha_A||_\varphi^2\\
			& \leq 2^{2n}C \sum_A ||D^*_\varphi\alpha_A||_\varphi^2.
		\end{align*}
		
		(Sufficiency) Suppose that the inequality (\ref{cond2}) is true. Let 
		\begin{align*}
			E = \{\sum_{\text{finite \ sum}}\lambda_jD^*_\varphi \beta_j; \beta_j\in C_0^\infty(\Omega,\mathbb R),\lambda_j\in \mathcal A_n\}\subset L^2_\varphi(\Omega, \mathcal A_n).
		\end{align*}
		It is easy to see that $E$ is an abelian group and hence a submodule over $\mathcal A_n$ by Definition \ref{defmodule}
		Recall that we aim to find solutions of $Du=f$. Define an $\mathcal A_n$-linear functional $T_f$ on $E$ as follows,
		\begin{align*}
			\langle T_f,\sum\lambda_j D^*_\varphi \beta_j\rangle  = \sum(\lambda_j f,\beta_j)_\varphi =\sum \lambda_j\int_\Omega \beta_j(x) f(x)e^{-2\varphi(x)} dx,
		\end{align*}
where $T_f$ is well-defined by \eqref{cond2}.
Note that from (\ref{cond2}) we have
		\begin{align*}
\begin{split}
|\langle T_f,\lambda D^*_\varphi \beta \rangle| &= |(\lambda f,\beta)_\varphi|=|(f,\overline\lambda \beta)_\varphi|\\
& \leq 2^n\sqrt C \sqrt{\sum_A \|D^*_\varphi(\overline\lambda \beta)_A\|^2_\varphi }=2^n\sqrt C \sqrt{\sum_A \|D^*_\varphi(\lambda_A \beta)\|^2_\varphi }\\
&=2^n\sqrt C\sqrt{\sum_A \|2\lambda_A\beta\overline D\varphi-\lambda_A\overline D\beta\|^2_\varphi   }=2^n\sqrt C\sqrt{\sum_A|\lambda_A|^2 \|D^*_\varphi\beta\|^2_\varphi  }\\
&=2^n\sqrt C|\lambda|||D^*_\varphi \beta||_\varphi=2^n\sqrt C ||\lambda D^*_\varphi\beta||_\varphi,
\end{split}
		\end{align*}
where $\beta\in C_0^\infty(\Omega,\mathbb R)$. The last equality follows from \eqref{e-121} and the fact that $D^*_\varphi\beta$ is a vector-valued function.
		By \cite[Proposition 2.13]{Brackx-Delanghe-Sommen}, the above inequality holds true on the whole set $E$, which shows that $T_f$ is bounded on $E$. Then, by the Hahn-Banach theorem in the Clifford algebra setting (cf. Theorem \ref{HB}), $T_f$ can be extended to a linear functional $\widetilde T_f$ on $L^2_\varphi(\Omega,\mathcal A_n)$ with
		\begin{align*}
			|\langle \widetilde T_f, g\rangle| \leq 2^n\sqrt{C} ||g||_\varphi, \quad \forall g\in L^2_\varphi(\Omega, \mathcal A_n).
		\end{align*} 
		Using the Riesz representation theorem, we deduce that there exists a $u\in L^2_\varphi(\Omega,\mathcal A_n),$ such that 
		\begin{align*}
			\langle \widetilde T_f,g\rangle = (u,g)_\varphi, \quad \forall g\in L^2_\varphi(\Omega,\mathcal A_n).
		\end{align*}
		For all $\beta\in C_0^\infty(\Omega,\mathbb R),$ let $g=D^*_\varphi \beta$. Then
		\begin{align*}
			(f,\beta)_\varphi = \langle \widetilde T_f, D^*_\varphi \beta \rangle = (u,D^*_\varphi\beta)_\varphi = (Du,\beta)_\varphi,
		\end{align*}
		which means that
		\begin{align*}
			\int_\Omega \beta(x) f(x)e^{-2\varphi(x)} dx = \int_\Omega \beta(x) Du(x)e^{-2\varphi(x)}dx.
		\end{align*}
		In particular we set $\beta(x)=\alpha(x)e^{2\varphi(x)},$ and then
		\begin{align*}
			\int_\Omega \alpha(x)f(x)dx = \int_\Omega \alpha(x) Du(x)dx, \quad \forall \alpha\in C_0^\infty(\Omega,\mathbb R).
		\end{align*}
		Therefore 
		\begin{align*}
			Du = f
		\end{align*}
		with $||u||_\varphi^2\leq 2^{2n}C$.
	\end{proof}
	
	\begin{proof}[{\rm \textbf{Proof of Theorem \ref{thm2}:}}]
		First we have
		\begin{align*}
			||D^*_\varphi \alpha_A||_\varphi^2 & = \int_\Omega |D^*_\varphi \alpha_A(x)|^2 e^{-2\varphi(x)} dx\\
			& = (D^*_\varphi\alpha_A,D^*_\varphi\alpha_A)_\varphi\\
			& = (\alpha_A, D(\alpha_A \overline D(2\varphi)-\overline D\alpha_A))_\varphi\\
			& = (\alpha_A,D\alpha_A \overline D(2\varphi))_\varphi + (\alpha_A, \alpha_A\laplace {(2\varphi)})_\varphi - (\alpha_A,\laplace \alpha_A)_\varphi\\
			& = (\alpha_A, D\alpha_A \overline D(2\varphi) - D\overline D\alpha_A)_\varphi + (\alpha_A,\alpha_A\laplace (2\varphi))_\varphi.
		\end{align*}
The second equality holds true since $D^*_\varphi\alpha_A$ is a vector-valued function. Note that
		\begin{align*}
			(\alpha_A,D\alpha_A\overline D(2\varphi))_\varphi &= \int_\Omega \alpha_A(x) D\alpha_A(x) \overline D(2\varphi(x))e^{-2\varphi(x)}dx\\
			& = -\int_\Omega (D\alpha_A(x)) \alpha_A(x) (\overline D e^{-2\varphi(x)})dx\\
			& = \int_\Omega(D\alpha_A(x)) (\overline D\alpha_A(x)) e^{-2\varphi(x)}dx + \int_\Omega \alpha_A(x)\laplace \alpha_A(x) e^{-2\varphi(x)} dx.
		\end{align*}
		Then we have
		\begin{align*}
			||D_\varphi^* \alpha_A||_\varphi^2 = ||\overline D\alpha_A||_\varphi^2 + \int_\Omega|\alpha_A(x)|^2 \laplace (2\varphi (x))e^{-2\varphi(x)}dx.
		\end{align*}
		
		Consequently we have
		\begin{align*}
			|(f,\alpha)_\varphi|^2 &\leq \left(\sum_A e_A|(f,\alpha_A)_\varphi|\right)^2\\
			&\leq 2^{2n} \sum_A |(f,\alpha_A)_\varphi|^2\\
			&= 2^{2n} \sum_A \left(\int_\Omega \frac{f(x)}{\sqrt{\laplace (2\varphi(x))}}\alpha_A(x)\sqrt{\laplace (2\varphi(x))} e^{-2\varphi(x)}dx\right)^2\\
			&\leq 2^{2n} \sum_A \left(\int_\Omega \frac{|f(x)|^2}{\laplace (2\varphi(x))}e^{-2\varphi(x)} dx \right) \left(\int_\Omega |\alpha_A(x)|^2\laplace (2\varphi(x)) e^{-2\varphi(x)}dx\right)\\
			& = 2^{2n} \sum_A||\frac{f}{\sqrt{\laplace (2\varphi)}}||_{\varphi}^2 ||\alpha_A \sqrt{\laplace (2\varphi)}||_\varphi^2\\
			&\leq 2^{2n}||\frac{f}{\sqrt{\laplace (2\varphi)}}||_{\varphi}^2 \sum_A||D^*_\varphi\alpha_A||_\varphi^2
		\end{align*}
		This completes the proof of Theorem \ref{thm2}.
	\end{proof}
	
	\subsection{Moser's iteration}
	In this section we prove Theorem \ref{thm3}, the weighted mean value inequality for $u\in F_\varphi^2(\mathbb R^{n+1},\mathcal{A}_n),$ which plays a crucial role in the estimate of the Bergman kernel.
	
	To prove this theorem, we use a modified argument of the Moser iteration method which is an important technique in the theory of partial differential equations (see e.g. \cite{Han-Lin}).
	
	\begin{proof}[{\rm \textbf{Proof of Theorem \ref{thm3}:}}]
		Since $u$ is left monogenic, we have
		\begin{align*}
			Du=0,\quad \text{ and } \laplace u_A= 0 \quad \text{for all possible index $A$}.
		\end{align*}
		In the following, for simplicity, we set $v=u_A,$ and write $$\int_{\mathbb R^{n+1}} v(x) dx = \int v $$ if there is no confusion.
		As given in \cite{Han-Lin}, we consider Moser's iteration for such $v.$
		Set $\tilde v = v^++k,$ where $k>0$ is a constant, $v^+= \max\{0,v\}$ and 
		\begin{align*}
			\tilde v_m=
			\begin{cases}
				\tilde v, & \mbox{if } v<m,\\
				k+m, & \mbox{if } v\geq m.
			\end{cases}
		\end{align*}
		Then we have $\del \tilde v_m=0$ in $\{v<0\}$ and $\{v>m\}$ and $\tilde v_m\leq \tilde v.$ Let
		\begin{align*}
			\phi=\eta^2(\tilde v_m^\beta\tilde v-k^{\beta+1}) 
		\end{align*}
		for some $\beta\geq 0$ and some nonnegative function $\eta\in C^1_0(B_0(1)).$
		Then
		\begin{align*}
			\del \phi = \eta^2\tilde v^\beta_m(\beta \del \tilde v_m +\del\tilde v) + 2\eta\del \eta(\tilde v^\beta_m\tilde v-k^{\beta+1}).
		\end{align*}
		Note that in $\{v\leq 0\}$ we have $\phi=0$ and $\del \phi=0$. It follows easily from the definitions of $\tilde v$ and $\tilde v_m$ that $v^+\leq \tilde v$ and $0\leq \tilde v^\beta_m\tilde v-k^{\beta+1}\leq \tilde v^\beta_m\tilde v$. Then
		\begin{align*}
			\int \langle \del v,\del \phi\rangle e^{-(\beta+2)\varphi}&= \int \beta\eta^2\tilde v^\beta_m \langle \del v,\del \tilde v_m \rangle e^{-(\beta+2)\varphi} +\int \eta^2\tilde v_m^\beta \langle \del v,\del\tilde v \rangle e^{-(\beta+2)\varphi}\\
			&\ \ \  + 2\int \langle \del v, \del \eta\rangle \eta (\tilde v_m^\beta\tilde v-k^{\beta+1})e^{-(\beta+2)\varphi}\\
			&=\int \beta \eta^2\tilde v_m^\beta |\del \tilde v_m|^2 e^{-(\beta+2)\varphi}+ \int \eta^2\tilde v_m^\beta |\del \tilde v |^2 e^{-(\beta+2)\varphi}\\
			&\ \ \ +2\int \langle \del v,\del \eta \rangle \eta(\tilde v_m^\beta\tilde v-k^{\beta+1})e^{-(\beta+2)\varphi}\\
			&\geq \beta \int \eta^2\tilde v_m^\beta |\del \tilde v_m|^2 e^{-(\beta+2)\varphi}+ \int \eta^2\tilde v_m^\beta |\del \tilde v |^2 e^{-(\beta+2)\varphi}\\
			&\ \ \  -\alpha_1^2 \int |\del \tilde v|^2 \tilde v_m^\beta \eta^2e^{-(\beta+2)\varphi}- \frac{1}{\alpha_1^2}\int |\del \eta|^2 \tilde v_m^\beta \tilde v^2e^{-(\beta+2)\varphi},
		\end{align*}
		where in the last inequality we have used $2\langle \del v,\del \eta\rangle \geq -(\alpha_1^2|\del v|^2+\frac{1}{\alpha^2_1}|\del \eta|^2).$
		This implies that
		\begin{equation}\label{tem_eq1}
		\begin{aligned}
				& \beta\int  \eta^2\tilde v_m^\beta |\del \tilde v_m|^2e^{-(\beta+2)\varphi} + (1-\alpha^2_1)\int \eta^2\tilde v_m^\beta |\del \tilde v |^2 e^{-(\beta+2)\varphi}\\
				&\leq  \int \langle \del v,\del \phi\rangle e^{-(\beta+2)\varphi} + \frac{1}{\alpha_1^2}\int |\del \eta|^2 \tilde v_m^\beta \tilde v^2 e^{-(\beta+2)\varphi}.
				\end{aligned}
				\end{equation}
				Then by integration by parts, 
				\begin{equation*}
				\begin{aligned}
			(\ref{tem_eq1})	& = -\int \phi\laplace v e^{-(\beta+2)\varphi} +(\beta+2) \int \phi\langle \del v,\del \varphi\rangle e^{-(\beta+2)\varphi} + \frac{1}{\alpha_1^2}\int |\del \eta|^2 \tilde v_m^\beta \tilde v^2 e^{-(\beta+2)\varphi}\\
				& = (\beta+2)\int \phi \langle \del v,\del \varphi\rangle e^{-(\beta+2)\varphi} + \frac{1}{\alpha_1^2}\int |\del \eta|^2 \tilde v_m^\beta \tilde v^2e^{-(\beta+2)\varphi}\\
				& = (\beta+2)\int \eta^2(\tilde v_m^\beta\tilde v-k^{\beta+1})\langle \del v,\del \varphi\rangle e^{-(\beta+2)\varphi} +\frac{1}{\alpha_1^2}\int |\del \eta|^2 \tilde v_m^\beta \tilde v^2 e^{-(\beta+2)\varphi}\\
				& \leq \frac{(\beta+2)}{2}\alpha_2^2 \int |\del \tilde v|^2\eta^2 \tilde v_m^\beta e^{-(\beta+2)\varphi}+ \frac{(\beta+2)}{2\alpha_2^2}\int |\del \varphi|^2\tilde v^2\tilde v_m^\beta\eta^2e^{-(\beta+2)\varphi}+\frac{1}{\alpha_1^2}\int |\del \eta|^2 \tilde v_m^\beta\tilde v^2e^{-(\beta+2)\varphi},
			\end{aligned}
		\end{equation*}
	where in the last inequality we have used $2\langle \del v,\del \varphi\rangle \leq (\alpha_2^2|\del v|^2+\frac{1}{\alpha^2_2}|\del \varphi|^2).$
		Consequently,
		\begin{align*}
			&\beta\int  \eta^2\tilde v_m^\beta |\del \tilde v_m|^2e^{-(\beta+2)\varphi} + (1-\alpha^2_1-\frac{(\beta+2)}{2}\alpha^2_2)\int \eta^2\tilde v_m^\beta |\del \tilde v |^2 e^{-(\beta+2)\varphi}\\
			&\leq  \frac{(\beta+2)}{2\alpha_2^2}\int |\del \varphi|^2\tilde v^2\tilde v_m^\beta\eta^2e^{-(\beta+2)\varphi} +\frac{1}{\alpha_1^2}\int |\del \eta|^2 \tilde v_m^\beta\tilde v^2e^{-(\beta+2)\varphi}.
		\end{align*}
		In particular we set $\alpha_1^2=\frac{1}{4}$ and $\alpha^2_2=\frac{1}{2(\beta+2)},$ and then we have
		\begin{align*}
			\beta\int  \eta^2\tilde v_m^\beta |\del \tilde v_m|^2 e^{-2\varphi} + \frac{1}{2}\int \eta^2\tilde v_m^\beta |\del \tilde v |^2e^{-2\varphi} \leq  (\beta+2)^2\int |\del \varphi|^2\tilde v^2\tilde v_m^\beta\eta^2 e^{-2\varphi} +4\int |\del \eta|^2 \tilde v_m^\beta\tilde v^2 e^{-2\varphi}.
		\end{align*}
		Set $w=\tilde v^\frac{\beta}{2}_m\tilde v.$ A direct computation gives
		\begin{align*}
			\del w &= \tilde v\del \tilde v^\frac{\beta}{2}_m + \tilde v_m^\frac{\beta}{2} \del \tilde v\\
			& = \frac{\beta}{2}\tilde v \tilde v^\frac{\beta-2}{2}_m \del \tilde v_m + \tilde v^\frac{\beta}{2}_m \del \tilde v,
		\end{align*}
		and then,
		\begin{align*}
			|\del w|^2 & = \langle \frac{\beta}{2}\tilde v \tilde v^\frac{\beta-2}{2}_m \del \tilde v_m + \tilde v^\frac{\beta}{2}_m \del \tilde v, \frac{\beta}{2}\tilde v \tilde v^\frac{\beta-2}{2}_m \del \tilde v_m + \tilde v^\frac{\beta}{2}_m \del \tilde v \rangle \\
			& = \frac{\beta^2}{4} \tilde v^2 \tilde v^{\beta-2}_m |\del \tilde v_m|^2 + \tilde v^\beta_m|\del \tilde v|^2 + \beta \tilde v \tilde v^{\beta-1}_m \langle \del\tilde v_m,\del \tilde v \rangle\\
			&=\frac{\beta^2}{4} \tilde v^{\beta}_m|\del \tilde v_m|^2+ \tilde v^\beta_m|\del \tilde v|^2 +\beta \tilde v^{\beta}_m\langle \del\tilde v_m,\del \tilde v \rangle\\
			&\leq \frac{\beta^2}{4} \tilde v_m^\beta |\del \tilde v_m|^2 +\tilde v^\beta_m|\del\tilde v|^2 + \beta \tilde v_m^\beta |\del \tilde v|^2 + \frac{\beta}{4}\tilde v_m^\beta |\del \tilde v_m|\\
			&= \frac{(\beta+1)\beta}{4}\tilde v_m^\beta|\del\tilde v_m|^2 + (\beta+1)\tilde v_m^\beta |\del \tilde v|^2\\
			&\leq 2(\beta+1)\left(\beta \tilde v_m^\beta |\del \tilde v_m|^2 +\frac{1}{2}\tilde v_m^\beta |\del\tilde v_m|^2\right),
		\end{align*}
		where the third equality follows from $\del \tilde v_m=0$ in $\{v< 0\}$ and $\{v> m\}.$
		
		Therefore we have
		\begin{align*}
			\int |\del w|^2 \eta^2 e^{-(\beta+2)\varphi}&\leq 2(1+\beta)\left(\beta\int \tilde v_m^\beta |\del \tilde v_m|^2\eta^2 e^{-(\beta+2)\varphi}+ \frac{1}{2}\int \tilde v_m^\beta |\del \tilde v|^2\eta^2e^{-(\beta+2)\varphi}\right)\\
			&\leq 2{(1+\beta)} \left((\beta+2)^2\int |\del \varphi|^2\tilde v^2\tilde v_m^\beta\eta^2 e^{-(\beta+2)\varphi} +4\int |\del \eta|^2 \tilde v_m^\beta\tilde v^2 e^{-(\beta+2)\varphi}\right)\\
			&=2{(1+\beta)} \left((\beta+2)^2\int |\del\varphi|^2 w^2\eta^2 e^{-(\beta+2)\varphi}+4\int |\del \eta|^2 w^2e^{-(\beta+2)\varphi}\right).
		\end{align*}
		
		Note that 
		\begin{align*}
			\int |\del (w\eta e^{-\frac{\beta+2}{2}\varphi})|^2 &\leq 3 \left(\int |\del w|^2\eta^2e^{-(\beta+2)\varphi}+\int |\del \eta|^2w^2e^{-(\beta+2)\varphi} + \frac{(\beta+2)^2}{4} \int |\del \varphi|^2 w^2\eta^2 e^{-(\beta+2)\varphi}\right)\\
			&\leq 3\left(2(1+\beta)(\beta+2)^2+\frac{(\beta+2)^2}{4}\right)\int  |\del \varphi|^2w^2\eta^2e^{-(\beta+2)\varphi}\\
			& \ \ \  +24(1+\beta) \int |\del \eta|^2w^2e^{-(\beta+2)\varphi}\\
			&\leq 24 (1+\beta)(\beta+2)^2\left(\int |\del \eta|^2w^2e^{-(\beta+2)\varphi}+ \int |\del \varphi|^2 w^2\eta^2 e^{-(\beta+2)\varphi}\right)
		\end{align*}
		
		By Sobolev's inequality we have
		\begin{align*}
			\left(\int |w\eta e^{-\frac{\beta+2}{2}\varphi}|^{2\chi}\right)^\frac{1}{\chi}\leq \widetilde C(\beta+2)^3 \left(\int (|\del \varphi|^2\eta^2 + |\del \eta|^2) w^2e^{-(\beta+2)\varphi}\right),
		\end{align*}
		where $\chi=\frac{n+1}{n-1}$ for $n>1$ and $\chi>2$ for $n=1.$ The cut-off function is chosen as follows. For any $0<r<R\leq 1$, set $\eta\in C_0^1(B_0(R))$ with the property 
		\begin{align*}
			\eta \equiv 1 \text{ in }B_0(r) \text{ and } |\del \eta|\leq \frac{2}{R-r}.
		\end{align*}
		Note that by assumption we have
		\begin{align*}
			|\del \varphi(x)|\leq L{ |x|},
		\end{align*}
		where $L$ is a positive constant. 
		Then 
		\begin{align*}
			\left(\int_{B_0(r)} w^{2\chi} e^{-(\beta+2)\chi\varphi}\right)^\frac{1}{\chi} \leq C\frac{(\beta+2)^3}{(R-r)^2}\int_{B_0(R)} w^2 e^{-(\beta+2)\varphi},
		\end{align*}
		which is equivalent to 
		\begin{align*}
			\left(\int_{B_0(r)} \tilde v^{\beta\chi}_m \tilde v^{2\chi}e^{-(\beta+2)\chi\varphi}\right)^\frac{1}{\chi}\leq C\frac{(\beta+2)^3}{(R-r)^2}\int_{B_0(R)}\tilde v_m^\beta\tilde v^2e^{-(\beta+2)\varphi}.
		\end{align*}
		Consequently,
		\begin{align*}
			\left(\int_{B_0(r)}\tilde v_m^{\gamma\chi}e^{-\gamma\chi\varphi}\right)^\frac{1}{\gamma\chi}\leq \left(C\frac{(\beta+2)^3}{(R-r)^2}\right)^\frac{1}{\gamma}\left(\int_{B_0(R)} \tilde v^{\gamma}e^{-\gamma\varphi}\right)^\frac{1}{\gamma},
		\end{align*}
		where $\gamma=\beta+2.$ Letting $m\to \infty,$ we have
		\begin{align*}
			\left(\int_{B_0(r)}\tilde v^{\gamma\chi}e^{-\gamma \chi\varphi}\right)^\frac{1}{\gamma\chi}\leq \left(C\frac{\gamma^3}{(R-r)^2}\right)^\frac{1}{\gamma}\left(\int_{B_0(R)} \tilde v^{\gamma}e^{-\gamma\varphi}\right)^\frac{1}{\gamma}.
		\end{align*}
		As in \cite{Han-Lin}, we iterate, beginning with $\gamma =2,$ as $2,2\chi,2\chi^2,...$. Set 
		\begin{align*}
			\gamma_i=2\chi^i, \quad \text{ and } r_i=\frac{1}{2}+\frac{1}{2^{i+1}}
		\end{align*}
		for $i=0,1,2,...$.
		Then by $\gamma_i=\gamma_{i-1} \chi,$ and $r_{i-1}-r_i=\frac{1}{2^{i+1}}$, we have
		\begin{align*}
			||\tilde v e^{-\varphi}||_{L^{\gamma_1}(B_{0}(r_1))}&\leq \left(2^{2(1+1)}C{\gamma_0^3}\right)^\frac{1}{\gamma_0}||\tilde v e^{-\varphi}||_{L^{\gamma_0}(B_0({r_0}))},\\
			||\tilde v e^{-\varphi}||_{L^{\gamma_2}(B_0({r_2}))}&\leq \left(2^{2(2+1)}C{\gamma_1^3}\right)^\frac{1}{\gamma_1}||\tilde v e^{-\varphi}||_{L^{\gamma_1}(B_0({r_1}))},\\
			&\ .\\
			&\ .\\
			&\ .\\
			||\tilde v e^{-\varphi}||_{L^{\gamma_i}(B_0({r_i}))}&\leq \left(2^{2(i+1)}C{\gamma_i^3}\right)^\frac{1}{\gamma_{i-1}}||\tilde v e^{-\varphi}||_{L^{\gamma_{i-1}}(B_0({r_{i-1}}))}.
		\end{align*}
		Using the argument in \cite{Han-Lin}, we can find a certain constant $M,$ which is independent of $\gamma_i,$ such that
		\begin{align*}
			||\tilde v e^{-\varphi}||_{L^{\gamma_i}(B_0({r_i}))}\leq M^{\sum_{j=0}^{i-1}\frac{j}{\gamma_{j}}} ||\tilde v e^{-\varphi}||_{L^{\gamma_{0}}(B_0({r_{0}}))},
		\end{align*}
		and consequently, 
		\begin{align*}
			\left(\int_{B_0({\frac{1}{2}})} \tilde v^{2\chi^i}e^{-2\chi^i\varphi}\right)^\frac{1}{2\chi^i}\leq \widetilde{ M} \left(\int_{B_0(1)}\tilde v^2 e^{-2\varphi}\right)^\frac{1}{2}.
		\end{align*}
		Let $i\to\infty,$ we have
		\begin{align*}
			\sup_{B_0(\frac{1}{2})} \tilde v e^{-\varphi}\leq \widetilde{ M} ||\tilde v e^{-\varphi}||_{L^2(B_0(1))}.
		\end{align*}
		Then
		\begin{align*}
			\sup_{B_0(\frac{1}{2})} v^+ e^{-\varphi}\leq \widetilde{ M} ||v^+ e^{-\varphi}||_{L^2(B_0(1))}
		\end{align*}
		as $k\to 0.$ 
		Similarly, we apply the same argument to $-v^-$
		and deduce that
		\begin{align*}
			\sup_{B_0(\frac{1}{2})} (-v^- e^{-\varphi})\leq \widetilde M ||v^- e^{-\varphi}||_{L^2(B_0(1))}.
		\end{align*}
		Finally,
		\begin{align*}
			\sup_{B_0(\frac{1}{2})} |v|e^{-\varphi} = \sup_{B_0(\frac{1}{2})} (v^+-v^-)e^{-\varphi}\leq \sup_{B_0(\frac{1}{2})} v^+e^{-\varphi} +\sup_{B_0(\frac{1}{2})} (-v^-e^{-\varphi}) \leq 2\widetilde{ M} ||ve^{-\varphi} ||_{L^2(B_0(1))},
		\end{align*}
		which gives 
		\begin{align*}
		    |u_A(0)e^{-2\varphi(0)}|\leq 2\widetilde M \left(\int_{B_0(1)} |u_A(y)|^2 e^{-2\varphi(y)} dy\right)^\frac{1}{2}.
		\end{align*}
		By a rescaling argument, we have
		$$
			||ve^{-\varphi} ||_{L^\infty(B_0(\frac{R}{2}))} \leq {2\widetilde M R^{-\frac{n}{2}}} ||ve^{-\varphi} ||_{L^2(B_0(R))},
		$$
		where $0<R<1.$
		Consequently, by translation, we have
		\begin{align*}
			|u_A(x)e^{-\varphi(x)}| \leq C_R \left(\int_{B_x(R)} |u_A(y)|^2 e^{-2\varphi(y)} dy\right)^\frac{1}{2}.
		\end{align*}
		The proof is completed.
	\end{proof}

	\section{Estimate of the Bergman kernel}
	In this section, we are devoted to proving the main theorem.
	As shown before, we prove that the point-wise evaluation functional $T$ is bounded. Then by the Riesz representation theorem, there exists $B_\varphi(\cdot,x)\in F^2_\varphi(\mathbb R^{n+1},\mathcal{A}_n)$ such that
	\begin{align*}
		u(x) &= ( u,B_\varphi(\cdot,x))_\varphi\\
		& = \int_{\mathbb R^{n+1}} \overline {B_\varphi(y,x)} u(y) e^{-2\varphi(y)}dy
	\end{align*}
	for all $u\in F^2_\varphi(\mathbb R^{n+1},\mathcal{A}_n).$ Moreover, we have
	\begin{align*}
	    |B_\varphi(x,x)|=\sup_{||u||_\varphi=1} |u(x)|^2,
	\end{align*}
	which plays a crucial role in the study of Bergman kernel. In fact, 
	\begin{align*}
	    |B_\varphi(x,x)| & = \left[\int_{\mathbb R^{n+1}}\overline{B_\varphi(y,x)}B_\varphi(y,x)e^{-2\varphi(y)}dy\right]_0 \\
	    & = \int_{\mathbb R^{n+1}}|B_\varphi(y,x)|^2e^{-2\varphi(y)}dy\\
	    & = \sup_{||u||_\varphi=1}\left[\int_{\mathbb R^{n+1}}\overline{B_\varphi(y,x)}u(y)e^{-2\varphi(y)}dy\right]_0^2\\
	    & = \sup_{||u||_\varphi=1}|u(x)|^2,
	\end{align*}
	where the third equality follows from duality of $L^2$ space for general vector-valued functions.
	The kernel $B_\varphi(x,y)$ is left-monogenic in $x$ and anti-right-monogenic in $y$, so it is smooth.

	\subsection{Estimate of the weighted Bergman kernel on the diagonal}
	In this subsection we will give estimates of $B_\varphi(w,x)$ on the diagonal.
	\begin{prop}
		There exist positive constants $C_1,C_2$ such that 
		\begin{align*}
			|B_\varphi(x,x)|\leq  C_1e^{2\varphi(x)}
		\end{align*}
		and if additionally, $\varphi$ is $2$-homongenous,
		\begin{align*}
			|B_\varphi (x, x)|\geq C_2 e^{2\varphi(x)}.
		\end{align*}
	\end{prop}
	\begin{proof}
		First, by the weighted mean inequality we have 
		\begin{align*}
			|B_\varphi(z,x)|^2 e^{-2\varphi(z)} &\leq C_R^2 \int_{B_z(R)} |B_\varphi(w, x)|^2e^{-2\varphi(w)} dw\\
			&\leq C_R^2 \int_{\mathbb R^{n+1}} |B_\varphi(w,x)|^2 e^{-2\varphi(w)} dw\\
			&=C_R^2 \sup_{||u||_\varphi=1} |u(x)|^2\\
			&\leq C_R^4 e^{2\varphi(x)},
		\end{align*}
		
		where $C_R$ is the constant in Theorem \ref{thm3}.
		This shows that
		\begin{align*}
			|B_{\varphi} (z,x)| \leq C_R^2 e^{\varphi(x)+\varphi(z)}.
		\end{align*}
		
		To estimate the lower bound of $B_\varphi(x, x),$ we use a version of $L^2$-estimate method in the Clifford algebra setting.
		We first write 
		\begin{align*}
			\varphi(x) = h(x) + |x|^2 + q(x),
		\end{align*}
		where $h$ is a harmonic polynomial of degree less than or equal to two, $h(0)=\varphi(0)$ and $q(x)=o(|x|^2).$ In fact, one can reduce the formula of $\varphi(x)$ in the above form after a suitable linear transformation on $x$ centered at $0$ (cf. in the proof of
Theorem 10 in Lindholm's paper \cite{Lindholm}). We also have
		\begin{align*}
			|2\varphi(x)-2|x|^2-2h(x)|\leq c^2(\tau) \tau^2
		\end{align*}
		in $B_0(\tau),$ for some function $c(\tau)$ such that $c(0)=0$ and $c(\tau)$ is continuous and nondecreasing. Moreover, we can choose the unique $\tau=\tau(k)$ so that $c(\tau)\tau^2k^2=1$ for a given $k$. When $k$ is increasing, $c(\tau)\tau^2$ is decreasing. With the condition that $c(\tau)$ is nondecreasing, we deduce that $\tau$ is a strictly decreasing function of $k$, $\tau\to 0$ when $k\to\infty$. In $B_0(\tau)$ we get the estimate
		\begin{align*}
			|2k^2\varphi(x)-2k^2|x|^2-2k^2h(x)|\leq k^2c^2(\tau)\tau^2 = c(\tau).
		\end{align*}
		Choose $\chi\in C^\infty_c(\mathbb R^{n+1},\mathbb R),$ 
		such that supp $\chi\subset B_0(1)$, $\chi\equiv 1$ on $B_0(\frac{1}{2}).$
		Define $g(x)=k^\frac{n+1}{2}\chi(\frac{x}{\tau})e^{k^2h(x)}$. 
		Then 
		\begin{align*}
			||g||_{k^2\varphi}^2 &= \int_{\mathbb R^{n+1}} k^{n+1} e^{2k^2h(x)}\chi(\frac{1}{\tau}x) e^{-2k^2\varphi(x)}dx\\
			&\leq \int_{B_0(\tau)} k^{n+1} e^{2k^2h(x)} e^{-2k^2\varphi(x)}dx\\
			&\leq \int_{B_0(\tau)} k^{n+1} e^{-2k^2|x|^2-c(\tau)}dx\\
			&\leq e^{-c(\tau)}\int_{B_0(k\tau)} e^{-|x|^2} dx.
		\end{align*}
		Recall that $\tau\to 0$ when $k\to\infty$, then $c(\tau)\to c(0)=0$ and $k\tau=\frac{1}{\sqrt{c(\tau)}}\to\infty$.
		Hence 
		\begin{align*}
		||g||_{k^2\varphi}^2\to\int_{\mathbb{R}^{n+1}} e^{-|x|^2} dx=\left(\frac{\pi}{2}\right)^\frac{n+1}{2}.
		\end{align*}
		We also have $|D g(x)|^2 \leq Ck^{n+1}e^{2k^2h(x)}/\tau^2$ and $D g(x)= 0$ except when $\tau/2<|x|<\tau.$ Then using the analogue of H\" ormander's $L^2$-method in the Clifford algebra setting (cf. Theorem \ref{thm1}) 
		we can solve $Du = Dg$ with
		\begin{align*}
			\int_{\mathbb R^{n+1}} |u(x)|^2 e^{-2k^2\varphi(x)} dx &\leq C \int_{\tau/2<|x|<\tau} \frac{k^{n+1}}{\tau^2} e^{2k^2h(x)}\frac{1}{k^2} e^{-2k^2\varphi(x)} dx\\
			&\leq C \frac{k^{n+1}}{k^2\tau^2} \int_{\tau/2<|x|<\tau} e^{-2k^2|x|^2 +c(\tau)}dx\\
			&\leq C^\prime \pi^{n+1} \frac{e^{-k^2\tau^2/4}}{k^2\tau^2}e^{c(\tau)},
		\end{align*}
		which tends to $0$, because $k^2\tau^2 = 1/c(\tau)\to\infty$ as $k\to\infty$. Then by applying Theorem \ref{thm3} to $u$ we have
		\begin{align*}
			k^{-(n+1)} |u(0)|^2e^{-2k^2\varphi(0)} \leq M e^{c(\tau)} \int_{\mathbb R^{n+1}} |u(x)|^2e^{-2k^2\varphi(x)} dx \to 0.
		\end{align*}
		Hence $G(x)=g(x)-u(x)$ is left-monogenic and 
		\begin{align*}
			\frac{|G(0)|}{||G||_{k^2\varphi}} \geq \frac{|g(0)|-|u(0)|}{||g||_{k^2\varphi}+||u||_{k^2\varphi}}.
		\end{align*}
		By the estimates for $g$ and $u$ we have
		\begin{align*}
			\liminf_{k\to\infty} k^{-(n+1)}|B_{k^2\varphi}(0,0)|e^{-2k^2\varphi(0)} \geq \left(\frac{2}{\pi}\right)^\frac{n+1}{2}.
		\end{align*}
		Consequently, for $k$ sufficiently large, there exists a $\delta>0$ such that
		\begin{align*}
			k^{-(n+1)} |B_{k^2\varphi}(0,0)| e^{-2k^2\varphi(0)} \geq \left(\frac{2}{\pi}\right)^\frac{n+1}{2}-\delta >0,
		\end{align*}
		or equivalently, 
		\begin{align*}
			|B_{k^2\varphi}(0,0) | \geq \left(\left(\frac{2}{\pi}\right)^\frac{n+1}{2}-\delta\right) k^{n+1} e^{2k^2\varphi(0)}. 
		\end{align*}
		By translation, we have
		\begin{align*}
			|B_{k^2\varphi}(x,x)| \geq C k^{n+1}e^{2k^2\varphi(x)}.
		\end{align*}
		If additionally, $\varphi$ is $2$-homogeneous, we deduce that
		\begin{align*}
			|B_{\varphi}(kx,kx)| \geq C e^{2\varphi(kx)}, 
		\end{align*}
		or equivalently,
		\begin{align*}
			|B_{\varphi}(y,y)| \geq C e^{2\varphi(y)}, \quad y\in \mathbb R^{n+1}.
		\end{align*}
	\end{proof}
	
	\subsection{Estimate of the Bergman kernel off the diagonal}
	
	In the following we give an estimate of the Bergman kernel off the diagonal,
	after which we complete the proof of our main theorem.
	\begin{prop}\label{p-0128}
		There exist constants $C>0,\alpha>0$ such that
		\begin{equation*}
			|B_\varphi(x, y)|\leq Ce^{\varphi(x)+\varphi(y)-\alpha|x-y|}.
		\end{equation*}
	\end{prop}
	
	\begin{proof}
		By Theorem \ref{thm3}, we have
		\begin{equation}\label{eq1}
			|B_\varphi(y,x)|^2e^{-2\varphi(y)}\leq C_3\int_{B_{y}(1)}|B_\varphi( z,x )|^2e^{-2\varphi(z)}dz.
		\end{equation}
		When $|x-y|\leq 8$,
		\begin{equation*}
			|B_\varphi(y,x)|^2e^{-2\varphi(y)}\leq C_3\int_{\mathbb{R}^{n+1}}|B_\varphi(z,x)|^2e^{-2\varphi(z)}dz=|B_\varphi(x,x)|\leq C_4e^{2\varphi(x)}.
		\end{equation*}
		Then it is done.
		
		Now it suffices to consider the case when $|x-y|>8$.
		Let $\delta=|x-y|/2$. Let $\tau$ be a smooth function, where $\tau=1$ outside $B_x(\delta)$, $\tau=0$
		in $B_x(\delta/2)$ and $|D\tau|^2\leq\frac{\tau}{\delta^2}$.
		When $\gamma(x):=\int_{\mathbb R^{n+1}}|B_\varphi(z,x)|^2\tau(z)e^{-2\varphi(z)}dz>0$ for fixed $x$, we can set $f(z):={B_\varphi(z,x)}/\sqrt{\gamma(x)}$ and derive the following
		\begin{equation}\label{eq2}
			\begin{split}
				\int_{|x-z|>\delta}|B_\varphi(z,x)|^2e^{-2\varphi(z)} dz&\leq \int_{\mathbb R^{n+1}}|B_\varphi(z,x)|^2\tau(z)e^{-2\varphi(z)}dz\\
				&\leq\sup_{\int|f|^2\tau e^{-2\varphi}=1}\big| \int_{\mathbb R^{n+1}}\overline{ B_\varphi(z, x)}f(z)\tau(z)e^{-2\varphi(z)} dz \big|^2\\
				&=\sup_{\int|f|^2\tau e^{-2\varphi}=1}|P_\varphi(f\tau)(x)|^2,
			\end{split}
		\end{equation}
		where $f$ is left-monogenic and $P_\varphi$ is the orthogonal projection from $L^2_\varphi(\mathbb{R}^{n+1},\mathcal A_n)$ to $F^2_\varphi(\mathbb{R}^{n+1},\mathcal A_n)$ defined by
		\begin{equation*}
			P_\varphi(g)(x):=\int_{\mathbb R^{n+1}} \overline{B_\varphi(y, x)}g(y)e^{-2\varphi(y)}dy,\quad g\in L^2_\varphi(\mathbb R^{n+1},\mathcal{A}_n).
		\end{equation*}
		Let $u$ be the solution of the equation $D u=D(f\tau)$ with minimal norm in $L^2_\varphi(\mathbb{R}^{n+1},\mathcal A_n)$.
		Then $P_\varphi(f\tau)=f\tau-u$.
		It follows from \eqref{eq1}, \eqref{eq2} and $\tau(x)=0$ that
		\begin{equation}\label{eq3}
			|B_\varphi(y,x)|^2e^{-2\varphi(y)}\leq C_5\sup_{\int|f|^2\tau e^{-2\varphi}=1} |f(x)\tau(x)-u(x)|^2=C_5\sup_{\int|f|^2\tau e^{-2\varphi}=1} |u(x)|^2,
		\end{equation}
	where the first inequality follows from (\ref{eq1}) with
		\begin{align*}
		    |x-z|>|x-y|-|z-y|>|x-y|-\frac{|x-y|}{2}=\delta.
		\end{align*}
	Set $\rho(y):=\text{dist}(y,B_x(\delta/3))$. By Theorem \ref{thm2} we have
	\begin{equation}\label{eq4}
			\begin{split}
				|u(x)|^2e^{-2\varphi(x)}&\leq C_6\int_{B_x(1)}|u(z)|^2e^{-2\varphi(z)}dz\\
				&\leq C_7\int_{\mathbb R^{n+1}} e^{-2c_1\rho(z)}|u(z)|^2e^{-2\varphi(z)}dz\\
				&\leq C_7C_8\int_{\mathbb R^{n+1}}| D(e^{-c_1\rho(z)}u(z))|^2e^{-2\varphi(z)}dz\\
				&\leq 2C_7C_8\int_{\mathbb R^{n+1}}| De^{-c_1\rho(z)}|^2|u(z)|^2e^{-2\varphi(z)}dz\\
				&+2C_7C_8\int_{\mathbb R^{n+1}} e^{-2c_1\rho(z)}|f(z)|^2|D\tau|^2e^{-2\varphi(z)}dz.
			\end{split}
		\end{equation}
		The constant $c_1$ is determined such that  the following inequality holds
		\begin{equation*}
			| De^{-c_1\rho}|\leq \frac{1}{\sqrt{2C_8+1}}e^{-c_1\rho}.
		\end{equation*}
	Then we have 
		\begin{equation}\label{e-0701}
		\int_{\mathbb R^{n+1}} e^{-2c_1\rho(z)}|u(z)|^2e^{-2\varphi(z)}dz\leq \tilde C_8\int_{\mathbb R^{n+1}} e^{-2c_1\rho(z)}|f(z)|^2|D\tau|^2e^{-2\varphi(z)}dz.
		\end{equation}
		Hence by \eqref{eq3}, \eqref{eq4} and \eqref{e-0701} we obtain
		\begin{equation}\label{e-0702}
			\begin{split}
				|B_\varphi(y,x)|^2e^{-2\varphi(y)-2\varphi(x)}&\leq C_9\int_{\mathbb{R}^{n+1}}|u(z)|^2e^{-2\varphi(z)}e^{-2c_1\rho(z)}dz\\
				&\leq C_{10}\int_{\mathbb R^{n+1}} e^{-2c_1\rho(z)}|f(z)|^2|D\tau|^2e^{-2\varphi(z)}dz\\
				&\leq C_{10}\int_{\mathbb{R}^{n+1}\setminus B_x(\delta/2)} e^{-2c_1\rho(z)}|f(z)|^2\frac{\tau}{\delta^2}e^{-2\varphi(z)}dz.
			\end{split}
		\end{equation}
		Note that for $z\notin B_x(\delta/2)$, $\rho(z)\geq \alpha_1\delta$ for some constant $\alpha_1>0$.
		We deduce from \eqref{e-0702} that there exist constant $C>0$ and $\alpha>0$,
		\begin{equation}\label{e-0703}
			\begin{split}
			|B_\varphi(y,x)|^2e^{-2\varphi(y)-2\varphi(x)}&\leq C_{10}\frac{e^{-2c_1\alpha_1\delta}}{\delta^2}\int_{\mathbb{R}^{n+1}\setminus B_x(\delta/2)} |f(z)|^2 \tau e^{-2\varphi(z)}dz \\
			&=C_{10}\frac{e^{-2c_1\alpha_1\delta}}{\delta^2}\leq C^2e^{-4\alpha\delta}=C^2e^{-2\alpha |x-y|}.
			\end{split}
		\end{equation}
		Equivalently, we have
		\begin{equation*}
			|B_\varphi(y,x)|^2\leq C^2e^{2\varphi(y)+2\varphi(x)-2\alpha|x-y|}.
		\end{equation*}
		The proof is completed.
	\end{proof}
	
\begin{proof}[{\rm \textbf{Proof of Theorem \ref{thm4}:}}]	
	By the proof of Theorem 1.4, we have
	\begin{align}\label{mv_har}
			\begin{split}	|u(x)|^2e^{-2\varphi(x)} \leq & C_{R}^2 \int_{B_x(R)}|u(y)|^2e^{-2\varphi(y)}dy\\
				\leq & C^2_R\int_{\mathbb R^{n+1}}|u(y)|^2e^{-2\varphi(y)}dy.
			\end{split}
	\end{align}
	Then by the Riesz representation theorem, there exists $B_{\varphi,har}(\cdot,x)\in F^2_{\varphi,har}(\mathbb R^{n+1})$ such that 
	\begin{align*}
	    u(x) = (u,B_{\varphi,har}(\cdot, x))_\varphi =\int_{\mathbb R^{n+1}}\overline{B_{\varphi,har}(y,x)}u(y)e^{-2\varphi(y)}dy
	\end{align*}
	for all $u\in F^2_{\varphi,har}(\mathbb R^{n+1}).$ {Moreover, we have
	\begin{align*}
	    B_{\varphi,har}(x,x)=\sup_{||u||_{\varphi,har}=1} |u(x)|^2,
	\end{align*}
	 which follows from
	\begin{align*}
	    B_{\varphi,har}(x,x) 
	    & = \int_{\mathbb R^{n+1}}|B_{\varphi,har}(y,x)|^2e^{-2\varphi(y)}dy\\
	    & = \sup_{||u||_{\varphi,har}=1}\left|\int_{\mathbb R^{n+1}}\overline{B_{\varphi,har}(y,x)}u(y)e^{-2\varphi(y)}dy\right|^2\\
	    & = \sup_{||u||_{\varphi,har}=1}|u(x)|^2.
	\end{align*}
	Then by (\ref{mv_har}) we have 
		\begin{align*}
			|B_{\varphi,har}(z,x)|^2 e^{-2\varphi(z)} &\leq C_R^2 \int_{B_z(R)} |B_{\varphi,har}(w, x)|^2e^{-2\varphi(w)} dw\\
			&\leq C_R^2 \int_{\mathbb R^{n+1}} |B_{\varphi,har}(w,x)|^2 e^{-2\varphi(w)} dw\\
			&=C_R^2 \sup_{||u||_{\varphi,har}=1} |u(x)|^2\\
			&\leq C_R^4 e^{2\varphi(x)}.
		\end{align*}
	Hence, there exists a constant $C$ such that
	\begin{align*}
	    |B_{\varphi,har}(z,x)|\leq Ce^{\varphi(x)+\varphi(z)}.
	\end{align*}
	}
\end{proof}

\end{document}